\documentclass[12pt]{amsart}
\usepackage{amsthm}
\usepackage{amssymb}
\usepackage{cite}
\usepackage{cases}
\usepackage{longtable}
\usepackage{geometry}
\usepackage{enumerate}
\usepackage{setspace}
\usepackage{tikz-cd}
\usetikzlibrary{matrix, arrows}
\usepackage[unicode=true]
 {hyperref}
\hypersetup{
colorlinks=true,
urlcolor=blue,
citecolor=blue,
linkcolor=blue,
}

\allowdisplaybreaks
\newtheorem{thm}{Theorem}[section]
\newtheorem{prop}[thm]{Proposition}

\theoremstyle{definition}
\newtheorem{definition}[thm]{Definition}
\newtheorem{notation}[thm]{Notation}
\newtheorem{example}[thm]{Example}

\theoremstyle{remark}
\newtheorem{remark}[thm]{Remark}

\numberwithin{equation}{section}

\newcommand{\bZ}{\mathbb{Z}}

\newcommand{\GL}{\mathrm{GL}}
\newcommand{\Aut}{\mathrm{Aut}}
\newcommand{\op}{\mbox{\tiny op}}
\newcommand{\reg}{{\mbox{\tiny reg}}}

\makeatletter
\newcommand{\addresseshere}{%
  \enddoc@text\let\enddoc@text\relax
}
\makeatother

\begin{document}

\large 

\title{Representation theory of skew braces}
\author{Yuta Kozakai}
\address[Y. Kozakai]{Department of Mathematics\\
Tokyo University of Science\\
1-3 Kagurazaka, Shinjuku-ku\\
Tokyo\\
Japan 
}
\email{kozakai@rs.tus.ac.jp}
\author{Cindy (Sin Yi) Tsang}
\address[C. Tsang]{Department of Mathematics\\
Ochanomizu University\\
2-1-1 Otsuka, Bunkyo-ku\\
Tokyo\\
Japan}
\email{tsang.sin.yi@ocha.ac.jp}\urladdr{http://sites.google.com/site/cindysinyitsang/} 
\date{Last updated: \today}

\maketitle


\begin{abstract}According to Letourmy and Vendramin, a representation of a skew brace is a pair of representations on the same vector space, one for the additive group and the other for the multiplicative group, that satisfies a certain compatibility condition. Following their definition, we shall explain how some of the results from representation theory of groups, such as Maschke's theorem and Clifford's theorem, extend naturally to that of skew braces. We shall also give some concrete examples to illustrate that skew brace representations are more difficult to classify than group representations.
\end{abstract}

\tableofcontents


\section{Introduction}

In the study of set-theoretic solutions to the Yang-Baxter equation, Rump \cite{Rump} introduced an algebraic object called \textit{brace} and showed that braces correspond to the non-degenerate involutive solutions. As a generalization of this, Guarnieri and Vendramin \cite{GV} later defined \textit{skew brace} and showed that skew braces correspond to all non-degenerate solutions. One can then understand the non-degenerate solutions by studying skew braces. For example, by \cite{ann} it is known that whether a solution is multipermutation is related to the right nilpotency of the associated skew brace.

\vspace{3mm}

Let us recall the definition of brace and skew brace.

\begin{definition} A \emph{skew (left) brace} is a set $A=(A,\cdot,\circ)$ equipped with two group  operations $\cdot$ and $\circ$ such that the so-called brace relation
\[a\circ (b\cdot c) = (a\circ b)\cdot a^{-1} \cdot (a\circ c)\]
holds for all $a,b,c\in A$. For each $a\in A$, we shall write $a^{-1}$ for its inverse in $(A,\cdot)$ and $\overline{a}$ for its inverse in $(A,\circ)$. The groups $(A,\cdot)$ and $(A,\circ)$ are referred to as the \textit{additive group} and \textit{multiplicative group} of $A$, respectively. It is easy to see that $(A,\cdot)$ and $(A,\circ)$ share the same identity that we denote by $1$. A \textit{brace} is a skew brace with an abelian additive group. 
\end{definition}

Skew brace is a ring-like structure where the brace relation resembles the distributive law. In fact, brace was introduced as a generalization of radical rings \cite{Rump}. At the same time, skew brace may be regarded as an extension of groups by the following example.

\begin{example}Let $(A,\cdot)$ be any group. Then $(A,\cdot,\cdot)$ and $(A,\cdot,\cdot^{\op})$ are both skew braces, where $\cdot^{\op}$ is the opposite operation of $\cdot$ defined by $a\cdot^{\op} b = b\cdot a$ for all $a,b\in A$. Skew braces that arise in this way, respectively, are said to be \textit{trivial} and \textit{almost trivial} because they are essentially just groups.
\end{example}

It then seems natural to ask to what extent can concepts and results from group theory be generalized to skew braces. For example, identifying groups with almost trivial skew braces, one can view the binary operation
\[ a*b = a^{-1}\cdot (a\circ b) \cdot b^{-1}\]
in skew braces as an analog of the commutator in groups. There are natural analogs of center (e.g. socle and annihilator) and nilpotency (e.g. left, right, and strong nilpotency) for skew braces defined in terms of this operation $*$. Factorizations of skew braces were considered in \cite{factor, Ito} and there is an analog of It\^{o}'s theorem for skew braces. Isoclinism of skew braces was defined in \cite{isoclinism} and it was shown that every skew brace is isoclinic to a so-called stem skew brace. The theory of Schur covers for finite skew braces was developed in \cite{Schur cover} and it was shown that the Schur cover of a finite skew brace is unique up to isoclinism. As one can see from these examples, there are many similarities between groups and skew braces. There is no doubt that the integrations of \par\noindent techniques from group theory can help us better understand skew braces. 

\vspace{3mm}

In relation to Schur covers, the notion of (complex projective) representations of a skew brace was also introduced in \cite{Schur cover} (whose definition was based on \cite{Zhu}), but not much was said other than the definition. Inspired by \cite{Schur cover}, we decided to study skew brace representations further, and this is the purpose of our paper. In \S \ref{sec:repn1} and \S \ref{sec:repn2}, we describe some ways to construct skew brace representations from group representations, and explain how some of the results from representation theory of groups naturally extend to that of skew braces. In \S \ref{sec:ex}, we give some concrete examples to illustrate that skew brace representations are more difficult to classify than group representations.

\begin{notation} Throughout the paper, let $A = (A,\cdot,\circ)$ be a skew brace and let $k$ be a field. We shall also define
\[ \lambda_a,\, \lambda_a^{\op}  : A \longrightarrow A;\,\  \begin{cases}
 \lambda_a(b) = a^{-1}\cdot ( a\circ b)\\
 \lambda^{\op}_a(b) = (a\circ b)\cdot a^{-1}
\end{cases}\]
for each $a\in A$. They are automorphisms of $(A,\cdot)$ by the brace relation, and it is well-known that the maps
\[ \lambda,\, \lambda^{\op} : (A,\circ) \longrightarrow \Aut(A,\cdot); \,\
 \lambda(a)= \lambda_a,\,\
  \lambda^{\op}(a) = \lambda_a^{\op}\]
are group homomorphisms. Sometimes $\lambda$ is called the \emph{lambda map} of $A$. We note that $\lambda^{\op}$ is the lambda map of the skew brace $A^{\op} = (A,\cdot^{\op},\circ)$, which is also called the \emph{opposite skew brace} of $A$ as defined by \cite{op}. We shall write
\[ \Lambda_A = (A,\cdot)\rtimes_{\lambda^{\op}}(A,\circ)\]
for the natural outer semidirect product of $(A,\cdot)$ and $(A,\circ)$ given by $\lambda^{\op}$.
\end{notation}

Let us also recall the definition of ideal in the context of skew braces. It is the substructure that one needs in order to form skew brace quotients.

\begin{definition}A subset $I$ of $A$ is called an \emph{ideal} if it is a normal subgroup of both $(A,\cdot)$ and $(A,\circ)$, and $\lambda_a(I)\subseteq I$ for all $a\in A$. In this case, it is clear that $I = (I,\cdot,\circ)$ is a skew brace. We also have $a\cdot I = a\circ I$ for all $a\in A$, so one can naturally form a quotient skew brace $A/I$.
\end{definition}

\section{Linear algebraic perspective}\label{sec:repn1}

We first recall the definition of representation of a skew brace that is due to \cite[Remark 3.3]{Zhu} and \cite[Definition 4.1]{Schur cover}.


\begin{definition} A \emph{representation} of $A$ over $k$ is a triple $(V,\beta,\rho)$, where
\begin{enumerate}[(1)]
\item $V$ is a vector space over $k$;
\item $\beta : (A,\cdot)\longrightarrow \GL(V)$ is a representation of $(A,\cdot)$;
\item $\rho: (A,\circ)\longrightarrow \GL(V)$ is a representation of $(A,\circ)$;
\end{enumerate}
such that the relation
\begin{equation}\label{relation}
\beta(\lambda^{\op}_a(b)) = \rho(a)\beta(b)\rho(a)^{-1}
\end{equation}
holds for all $a,b\in A$. 
\end{definition}

We shall omit the phrase ``over $k$" for simplicity. All vector spaces, representations, etc. to be considered are over $k$.

\begin{remark}\label{key remark} The relation (\ref{relation}) is meant to mimic the conjugation action of the multiplicative group $(A,\circ)$ on the additive group $(A,\cdot)$ in the semidirect product $\Lambda_A$. This observation implies that given a representation $(V,\beta,\rho)$ of the skew brace $A$, we can naturally construct a representation $(V,\varphi_{(\beta,\rho)})$ of the group $\Lambda_A$ by defining
\[ \varphi_{(\beta,\rho)} : \Lambda_A \longrightarrow \GL(V);\,\ \varphi_{(\beta,\rho)}(a,b) = \beta(a)\rho(b).\]
We are grateful to Letourmy and Vendramin for pointing this out to us. In this way, a representation of the skew brace $A$ can be regarded as a representation of the group $\Lambda_A$.
\end{remark}

In the following examples, we shall describe a few ways to construct skew brace representations out of a group representation. The first methods (Examples \ref{ex:circ} and \ref{ex:cdot}) are to build a representation of $A$ from a representation of either $(A,\cdot)$ or $(A,\circ)$ by letting the other group act trivially.

\begin{example}\label{ex:circ} Let $(V,\rho)$ be any representation of the multiplicative group $(A,\circ)$. Then $(V,\beta_0,\rho)$ is a representation of $A$, where $\beta_0:(A,\cdot)\longrightarrow \GL(V)$ is the trivial representation on $V$.
\end{example}

\begin{example}\label{ex:cdot}Let $(V,\beta)$ be any representation of the additive group $(A,\cdot)$ such that $\beta(\lambda_a^{\op}(b)) = \beta(b)$ for all $a,b\in A$. Then $(V,\beta,\rho_0)$ is a representation of $A$, where $\rho_0:(A,\circ)\longrightarrow \GL(V)$ is the trivial representation on $V$.
\end{example}

In general, given a representation $(V,\beta)$ of the additive group $(A,\cdot)$, there need not exist $\rho:(A,\circ)\longrightarrow \GL(V)$ for which $(V,\beta,\rho)$ is a representation of $A$. For example, take $\dim_k(V)$ to be finite and let $\chi_\beta : (A,\cdot) \longrightarrow k$ denote the character of $\beta$. For such a $\rho$ to exist, by (\ref{relation}) we must have
\begin{equation}\label{chi}
\chi_\beta(\lambda_a^{\op}(b)) = \chi_\beta(b)
\end{equation}
for all $a,b\in A$. As the next example shows, this can fail to hold.

\begin{example}\label{ex:beta}
Consider the brace $(\mathbb{Z}/p\mathbb{Z}\times \mathbb{Z}/2\mathbb{Z},+,\circ)$, where
\[ \begin{pmatrix}a_1\\a_2\end{pmatrix} \circ \begin{pmatrix} b_1\\ b_2\end{pmatrix} = 
\begin{pmatrix}
a_1 + (-1)^{a_2}b_1\\ a_2+b_2\end{pmatrix}\]
for all $\left(\begin{smallmatrix}a_1\\a_2\end{smallmatrix}\right),\left(\begin{smallmatrix}b_1\\b_2\end{smallmatrix}\right)\in \bZ/p\bZ\times \bZ/2\bZ$, and $p$ is an odd prime (c.f. \cite[Theorem]{pq}). Suppose that the field $k$ contains a primitive $p$th root $\zeta_p$ of unity. Then for the representation $\beta : (\bZ/p\bZ\times \bZ/2\bZ,+)\longrightarrow \GL(k) \simeq k^\times$ defined by
\[ \beta\begin{pmatrix} 1\\0\end{pmatrix}= \zeta_p,
\,\ \beta\begin{pmatrix} 0\\1\end{pmatrix}=1,\]
we have $\chi_\beta(\lambda_{a}^{\op}(b))  = \zeta_p^{(-1)^{a_2}b_1}$ and $\chi_\beta(b) = \zeta_p^{b_1}$ for $a= \left(\begin{smallmatrix}a_1\\a_2\end{smallmatrix}\right),\, b=\left(\begin{smallmatrix}b_1\\b_2\end{smallmatrix}\right)$. Since they are not equal for $a_2=1,\, b_1=1$, we see that (\ref{chi}) is not satisfied.
\end{example}

Let $A*A$ be the subgroup of $(A,\cdot)$ generated by $\{a*b: a,b\in A\}$. Then $A*A$ is an ideal of $A$ (see \cite[Proposition 2.1]{perfect}); in fact the smallest ideal for which the quotient skew brace is trivial. The second methods (Examples \ref{ex:lift} and \ref{ex:lift'}) are to build a representation of $A$ by lifting a representation of the group
\[ A/A*A = (A/A*A,\cdot) = (A/A*A,\circ)\]
via the natural quotient map
 \[\pi_A : A \longrightarrow A/A*A;\,\ \pi_A(a) = a\cdot (A*A) = a\circ (A*A).\]
Note that $\pi_A$ is a group homomorphism on both $(A,\cdot)$ and $(A,\circ)$.

\begin{example}\label{ex:lift} Let $(V,\overline{\alpha})$ be any representation of the group $A/A*A$. Note that $(V,\alpha)$ is a representation of both $(A,\cdot)$ and $(A,\circ)$ for $\alpha = \overline{\alpha}\circ \pi_A$, and
\begin{align*}
\alpha(\lambda_a^{\op}(b)) & = \alpha((a\circ b )\cdot a^{-1})\\
& = \alpha( a\cdot b \cdot a^{-1}) \\
& = \alpha(a) \alpha(b)\alpha(a)^{-1}
\end{align*}
for all $a,b\in A$ because $\ker(\alpha)$ contains $A*A$. Hence, we see that $(V,\alpha,\alpha)$ satisfies (\ref{relation}) and is a representation of $A$.
\end{example}

\begin{example}\label{ex:lift'}Let $(V,\overline{\alpha})$ be any representation of the group $A/A*A$ such that $\mathrm{Im}(\alpha)$ is abelian. Then $(V,\alpha,\rho_0)$ and $(V,\beta_0,\alpha)$ are representations of $A$ for $\alpha = \overline{\alpha}\circ \pi_A$ by  Examples \ref{ex:cdot} and \ref{ex:circ}, respectively, where  $\rho_0$ and $\beta_0$ are the trivial representations on $V$. Note that
\[\alpha(\lambda_a^{\op}(b)) = \alpha(a) \alpha(b)\alpha(a)^{-1}=\alpha(b)\]
for all $a,b\in A$ because $\ker(\alpha)$ contains $A*A$ and $\mathrm{Im}(A)$ is abelian, whence Example \ref{ex:cdot} applies. It is easy to see that the tensor product 
\[(V\otimes_k V, \alpha\otimes\rho_0, \beta_0\otimes \alpha)\]
satisfies (\ref{relation}) and so is a representation of $A$.
\end{example}

In view of Remark \ref{key remark}, we make the following natural definitions.

\begin{definition}\label{def:X} A representation $(V,\beta,\rho)$ of $A$ is said to \textit{possess property $X$} if the representation $(V,\varphi_{(\beta,\rho)})$ of $\Lambda_A$ possesses property $X$, where $X$ is any property defined for a group representation, such as \textit{irreducible}, \textit{indecomposable}, and \textit{completely reducible}.
\end{definition}

\begin{definition}\label{def:equivalent} Two representations $(V_1,\beta_1,\rho_1)$ and $(V_2,\beta_2,\rho_2)$ of $A$ are said to be \textit{equivalent} if the representations $(V_1,\varphi_{(\beta_1,\rho_1)})$ and $(V_2,\varphi_{(\beta_2,\rho_2)})$ of $\Lambda_A$ are equivalent.
\end{definition}

Since every representation of $A$ may be viewed as a representation of $\Lambda_A$, certain results from the representation theory of groups naturally extend to that of skew braces. We give a couple examples here.

\begin{prop}\label{fd prop} Suppose that $A$ is finite. Then an irreducible representation of $A$ has degree at most $|A|^2$. 
\end{prop}
\begin{proof} It is well-known that the degree of an irreducible representation of a finite group $G$ is at most $|G|$. The claim then follows since $|\Lambda_A| = |A|^2$.
\end{proof}

\begin{thm}[Analog of Maschke's theorem]\label{Maschke thm}
Suppose that $A$ is finite and that $\mathrm{char}(k)$ is either $0$ or coprime to $|A|$. Then for any finite dimensional representation $(V,\beta,\rho)$ of $A$ over $k$, every $\varphi_{(\beta,\rho)}$-invariant subspace of $V$ admits a $\varphi_{(\beta,\rho)}$-invariant complement in $V$. Hence, by induction on the degree, every finite dimensional representation of $A$ over $k$ is completely reducible.
\end{thm}
\begin{proof}This follows from the usual Maschke's theorem because $|\Lambda_A| = |A|^2$, and the hypothesis implies that $\mathrm{char}(k)$ is either $0$ or coprime to $|\Lambda_A|$.
\end{proof}

It is well-known that Maschke's theorem admits a converse and it is proven by considering the regular representation.

\vspace{3mm}

Let $kA$
be the vector space over $k$ that has the set $A$ as a basis. Let
\[ \begin{cases}
\beta_\reg : (A,\cdot)\longrightarrow \GL(kA);& \beta_\reg(a)(b) = a\cdot b \mbox{ for all }a,b\in A,\\
\rho_\reg: (A,\circ)\longrightarrow \GL(kA);&\rho_\reg(a)(b) = a\circ b\mbox{ for all }a,b\in A,
\end{cases}\]
respectively, denote the left regular representations of $(A,\cdot)$ and $(A,\circ)$. As we shall verify below, the triplet $(kA,\beta_\reg,\rho_\reg)$ is a representation of $A$, and we shall refer to it as the \emph{left regular representation} of $A$.

\begin{prop}\label{prop:regular}The triplet $(kA,\beta_\reg,\rho_\reg)$ is a representation of $A$.
\end{prop}
\begin{proof}For any $a,b,c\in A$, we have 
\begin{align*}
\rho_\reg(a)\beta_\reg(b)\rho_\reg(a)^{-1}(c) & =\rho_\reg(a)\beta_\reg(b)\rho_\reg(\overline{a})(c)\\
& = a\circ (b\cdot (\overline{a}\circ c))\\
& = (a\circ b)\cdot a^{-1} \cdot (a\circ \overline{a}\circ c)\\
& = \lambda_{a}^{\op}(b)\cdot c\\
& = \beta_\reg({\lambda_a^{\op}(b))}(c).
\end{align*}
This verifies the relation (\ref{relation}) and the claim follows.
\end{proof}

It is natural to ask whether $(kA,\beta'_\reg,\rho'_\reg)$ is a representation of $A$, where
\[ \begin{cases}
\beta'_\reg : (A,\cdot)\longrightarrow \GL(kA);& \beta'_\reg(a)(b) = b\cdot a^{-1} \mbox{ for all }a,b\in A,\\
\rho'_\reg: (A,\circ)\longrightarrow \GL(kA);&\rho'_\reg(a)(b) = b\circ \overline{a}\mbox{ for all }a,b\in A,
\end{cases}\]
respectively, denote the right regular representations of $(A,\cdot)$ and $(A,\circ)$. It follows from (\ref{relation}) that $(kA,\beta'_\reg,\rho'_\reg)$ is a representation of $A$ if and only if
\begin{equation}\label{right regular}
 c\cdot (a\circ b^{-1})\cdot a^{-1}
=((c\circ a)\cdot b^{-1}) \circ \overline{a} \end{equation}
for all $a,b,c\in A$. As the next example shows, this does not hold in general.

\begin{example} Let $(\bZ/p\bZ\times \bZ/2\bZ,+,\circ)$ be the same brace in Example \ref{ex:beta}, where $p$ is an odd prime. Taking $c =a$, we may write (\ref{right regular}) as
\[ (a\circ a) - b = a\circ (-b)\circ a.\]
But letting $a = \left(\begin{smallmatrix} a_1 \\ a_2\end{smallmatrix}\right)$ and $b=\left(\begin{smallmatrix}b_1\\ b_2\end{smallmatrix}\right)$, we have
\begin{align*}
(a\circ a) - b & = \begin{pmatrix} a_1 + (-1)^{a_2}a_1-b_1\\ b_2\end{pmatrix},\\
a\circ (-b)\circ a& = \begin{pmatrix}a_1 + (-1)^{a_2}(-b_1) + (-1)^{a_2+b_2}a_1\\  b_2\end{pmatrix}.
\end{align*}
The first components are not equal when $a_2 = 1,\, b_2=0$, and $b_1\neq 0$.
\end{example}

Using the left regular representation of $A$ (which is not the same thing as the left regular representation of $\Lambda_A$!), we can prove that the analog of the converse of Maschke's theorem for skew brace representations holds.

\begin{thm}[Analog of converse of Maschke's theorem]\label{Maschke'} Suppose that $A$ is finite and that for any finite dimensional representation $(V,\beta,\rho)$ of $A$ over $k$, every $\varphi_{(\beta,\rho)}$-invariant subspace of $V$ admits a $\varphi_{(\beta,\rho)}$-invariant complement in $V$. Then $\mathrm{char}(k)$ is either $0$ or coprime to $|A|$.
\end{thm}
\begin{proof} Consider the left regular representation $(kA,\beta_\reg,\rho_\reg)$ of $A$. The subspace $\mathrm{span}_k(\delta)$ generated by $\delta = \sum_{a\in A} a$ is clearly $\varphi_{(\beta_\reg,\rho_\reg)}$-invariant, so the hypothesis implies that it admits a $\varphi_{(\beta_\reg,\rho_\reg)}$-invariant complement $U$ in $kA$. But $U$ is in particular $\beta_\reg$-invariant, so by the standard proof of the converse of the usual Maschke's theorem, we know that
\begin{equation}\label{U=}
 U =  \Bigg\{ \sum_{a\in A}c_aa \in kA  : \sum_{a\in A}c_a = 0\Bigg\}.
 \end{equation}
 Indeed, for any $\sum_{a\in A}c_aa\in U$, the $\beta_{\reg}$-invariance of $U$ implies that
\begin{align*}
\sum_{b\in A} \beta_\reg(b) \Bigg(\sum_{a\in A}c_aa \Bigg)
 = \sum_{a\in A} c_a \sum_{b\in A}ba
 = \Bigg(\sum_{a\in A}c_a\Bigg)\delta
 \end{align*}
also lies in $U$. Since $\mathrm{span}_k(\delta)\cap U=\{0\}$, we must have $\sum_{a\in A}c_a= 0$, which gives the left-to-right inclusion of (\ref{U=}), and equality follows as well by comparing dimensions. We deduce that $\mathrm{char}(k)$ is not a prime divisor of $|A|$, for otherwise $\delta\in U$ and $U$ cannot be a complement of $\mathrm{span}_k(\delta)$. 
\end{proof}

Theorem \ref{Maschke'} implies that for any finite skew brace $A$, the hypothesis on $\mathrm{char}(k)$ in Theorem \ref{Maschke thm} cannot be dropped. Let us end with an example of a finite dimensional representation of $A$ that is not completely reducible.

\begin{example}Consider the brace $A=(\mathbb{Z}/p\mathbb{Z}\times\mathbb{Z}/p\mathbb{Z},+,\circ)$, where
\[ \begin{pmatrix}a_1\\ a_2 \end{pmatrix} \circ \begin{pmatrix} b_1\\b_2\end{pmatrix} = \begin{pmatrix} a_1 + b_1 + a_2b_2 \\ a_2 + b_2\end{pmatrix}\]
for all $\left(\begin{smallmatrix}a_1\\ a_2\end{smallmatrix}\right),\left(\begin{smallmatrix} b_1\\ b_2\end{smallmatrix}\right)\in \bZ/p\bZ\times \bZ/p\bZ$, and $p$ is a prime (c.f. \cite[Proposition 2.4]{p^3}). Obviously $A*A= \mathbb{Z}/p\mathbb{Z} \times \{0\}$. Suppose that $\mathrm{char}(k)=p$ and consider the group representation $\overline{\alpha} : A/A*A\longrightarrow \mathrm{GL}_2(k)$ defined by
\[ \overline{\alpha}\left(\begin{pmatrix}0\\1\end{pmatrix} + (A*A)\right) = \overline{\alpha}\left(\begin{pmatrix}0\\1\end{pmatrix} \circ ( A*A)\right) = \begin{pmatrix}1 & 1 \\ 0 & 1 \end{pmatrix}. \]
Then $(k^2,\alpha,\alpha)$ is a representation of $A$ for $\alpha=\overline{\alpha}\circ \pi_A$ by Example \ref{ex:lift}, and it is not completely reducible because $\left(\begin{smallmatrix}1 & 1 \\ 0 & 1 \end{smallmatrix}\right)$ is not diagonalizable.
\end{example}

\section{Module theoretic perspective}\label{sec:repn2}

As in the case of groups, we may reinterpret skew brace representations in the language of modules. Here we use  the vector space $kA$ over $k$ with basis $A$. This has the advantage that we can employ tools from module theory. 

\begin{definition}A \emph{(left) $kA$-module} is a vector space $V$ over $k$ such that
\begin{enumerate}[(1)]
\item $V$ is a left $k(A,\cdot)$-module, where we write $\cdot$ for the action of $(A,\cdot)$;
\item $V$ is a left $k(A,\circ)$-module, where we write $\circ$ for the action of $(A,\circ)$; 
\end{enumerate}
and the two actions satisfy the relation
\begin{equation}\label{relation'}
\lambda_a^{\op}(b)\cdot v = a\circ (b \cdot (\overline{a} \circ v))
\end{equation}
for all $a,b\in A$ and $v\in V$.
\end{definition}

Clearly there is a one-to-one correspondence between representations of $A$ and $kA$-modules. Precisely, a representation $(V,\beta,\rho)$ of $A$ may be made into a $kA$-module by defining the actions to be
\[a\cdot v = \beta(a)(v),\,\ 
a\circ v = \rho(a)(v)\]
for all $a\in A$ and $v\in V$, and vice versa. Note that (\ref{relation'}) is simply a restatement of the relation (\ref{relation}) under this correspondence.

\begin{remark}\label{key remark'} By Remark \ref{key remark}, given a $kA$-module $V$, we can naturally turn it into a $k\Lambda_A$-module by letting $\Lambda_A$ act via
\[ (a,b) \star v = a\cdot (b\circ v) \]
for all $(a,b)\in \Lambda_A$ and $v\in V$. We shall denote by $\Phi(V)$ the vector space $V$ equipped with the $k\Lambda_A$-module structure given by $\star$. In this way, a module over $kA$ can be regarded as a module over the group ring $k\Lambda_A$.
\end{remark}

The following are counterparts of Definitions \ref{def:X} and \ref{def:equivalent}.

\begin{definition}  A $kA$-module $V$ is said to \textit{possess property $X$} if the $k\Lambda_A$-module $\Phi(V)$ possesses property $X$, where $X$ is any property defined for a module over a ring, such as \textit{simple}, \textit{indecomposable}, and \textit{semisimple}.
\end{definition}

\begin{definition} Two $kA$-modules $V_1$ and $V_2$ are said to be \textit{isomorphic} if the $k\Lambda_A$-modules $\Phi(V_1)$ and $\Phi(V_2)$ are isomorphic.
\end{definition}

With this module theoretic perspective, we can prove the following.

\begin{prop}\label{prop:simple} Suppose that $A$ is finite. Then every simple $kA$-module is semisimple as a $k(A,\cdot)$-module.
\end{prop}
\begin{proof} This is a consequence of Clifford's theorem (see \cite[Theorems 5.3.1]{Webb}) because a simple $kA$-module is in particular a simple $k\Lambda_A$-module, and $(A,\cdot)$ is a normal subgroup of $\Lambda_A$.
\end{proof}

In fact, Clifford's theorem admits a natural analog for $kA$-modules. Recall that ideal in skew braces is the natural analog of normal subgroup in groups (they are the substructures that allow one to form quotients).



\begin{thm}[Analog of Clifford's theorem]\label{clifford'} Suppose that $A$ is finite and let $I$ be an ideal of $A$. For any simple $kA$-module $S$, the restriction $\mathrm{Res}_I^A(S)$ of $S$ to $kI$ is always semisimple as a $kI$-module. In the case, we may write
\begin{equation}\label{decomposition}
 \mathrm{Res}_I^A(S) = \bigoplus_{i=1}^{r} S_i^{\oplus{m_i}},
 \end{equation}
where $r\in\mathbb{N}$, and $S_1,\dots,S_r$ are pairwise non-isomorphic simple $k\Lambda_I$-modules occurring with multiplicities $m_1,\dots,m_r$. We shall refer to $S_1^{\oplus m_1},\dots,S_r^{\oplus m_r}$ as the homogeneous components of $\mathrm{Res}_I^A(S)$. Then: 
\begin{enumerate}[$(a)$]
\item The group $\Lambda_A$ permutes the homogeneous components transitively.
\item We have $m_1=\cdots=m_r$ and $\dim_k(S_1) = \cdots = \dim_k(S_r)$.
\end{enumerate}
\end{thm}
\begin{proof} Below, we shall show that $\Lambda_I = (I,\cdot) \rtimes_{\lambda^{\op}} (I,\circ)$ is a normal subgroup of $\Lambda_A$. The claim would then follow immediately from Remark \ref{key remark'} and the usual Clifford's theorem (see \cite[Theorems 5.3.1 and 5.3.2]{Webb} for example).

\vspace{3mm}

 For any $(a,b)\in \Lambda_A$ and $(x,y) \in \Lambda_I$, we first observe that
 \begin{equation}\label{in I} \lambda_b(x),\, b\lambda_b(x)b^{-1},\, z, \, \overline{a}\circ z\circ a, \, \lambda_a(\overline{a}\circ z\circ a)\in I
 \end{equation}
 because $I$ is an ideal of $A$, where $z=b\circ y \circ \overline{b}$. Now, we have
\begin{align}\notag
(a,b) \cdot (x,y) \cdot (a,b)^{-1} & = 
(a\lambda^{\op}_b(x),b\circ y) \cdot (\lambda^{\op}_{\overline{b}}(a)^{-1},\overline{b})\\\notag
& = (a\lambda_b^{\op}(x)\lambda_{b\circ y\circ\overline{b}}^{\op}(a)^{-1},b\circ y\circ\overline{b})\\\label{conjugate}
& = (a\cdot b\lambda_b(x)b^{-1}\cdot z\lambda_z(a)^{-1}z^{-1}, z)
\end{align}
in the group $\Lambda_A$. By (\ref{in I}), modulo the normal subgroup $(I,\cdot)$, we have
\begin{align*}
a\cdot b\lambda_b(x)b^{-1}\cdot z\lambda_z(a)^{-1}z^{-1}
&\equiv a\lambda_z(a)^{-1}\\
&\equiv a(z\circ a)^{-1}z\\
 &\equiv a(a\circ \overline{a}\circ z \circ a)^{-1}\\
 &\equiv a (a\lambda_a(\overline{a}\circ z\circ a))^{-1}\\
 &\equiv aa^{-1}\\
 &\equiv1
\end{align*}
in the group $(A,\cdot)$. It then follows that (\ref{conjugate}) is element of $\Lambda_I$. Thus, indeed $\Lambda_I$ is a normal subgroup of $\Lambda_A$.
\end{proof}

In Theorem \ref{clifford'}(a), the additive group $(A,\cdot)$ on its own does not permute the homogeneous components transitively in general, and same for the multiplicative group $(A,\circ)$, so it is necessary to use the entire group $\Lambda_A$. Below, let us explain how to construct examples where $(A,\cdot)$ or $(A,\circ)$ does not act transitively on the homogeneous components.

\vspace{3mm}

Suppose that $A$ is almost trivial. Note that $\lambda_a^{\op}=\mathrm{id}_A$ and $\lambda_a$ is conjugation by $a$ for all $a\in A$ in this case. Let $\{\sharp,\sharp'\}=\{\cdot,\circ\}$. Then ideals of $A$ coincide with normal subgroups of $(A,\sharp)$. We start with any simple $k(A,\sharp)$-module $S$ and regard it as a $kA$-module by letting $(A,\sharp')$ act trivially (recall Examples \ref{ex:circ} and \ref{ex:cdot}). The $S_1,\dots,S_r$ in (\ref{decomposition}) are nothing but simple $k(I,\sharp)$-modules, and when $r\geq 2$, the group $(A,\sharp')$ by itself cannot permute the homogeneous components transitively because its action is trivial. 

\vspace{3mm}

In summary, it suffices to give an example of a finite group $G$ and a simple $kG$-module $S$ for which $\mathrm{Res}_N^G(S)$ has at least two homogeneous components (when decomposed as a direct sum of simple $kN$-modules) for some normal subgroup $N$ of $G$. There are plenty of such examples.
 
\begin{example}Let $G =S_n$ be the symmetric group for $n\in \{3,4\}$. Suppose that $ \mathrm{char}(k)\neq 3$ for $n=3$ and $ \mathrm{char}(k)\neq 2$ for $n=4$. Suppose also that $k$ has a primitive cube root $\omega$ of unity. Then for the unique simple $kS_n$-module $S$ of dimension two and $N=A_n$, there are two homogeneous components in the decomposition $\mathrm{Res}_{A_n}^{S_n}(S) = U_1\oplus U_2$, where $U_1$ and $U_2$ are one-dimensional $kA_n$-modules on which  $(1\, 2\, 3)$ acts as $\omega$ and $\omega^2$, respectively.
\end{example}

The next theorem is an analog of a well-known result from modular representation theory (see \cite[Proposition 6.2.1]{Webb} and Theorem 7.1 of version 1 of this paper on \href{https://arxiv.org/abs/2405.08662}{arXiv:2405.08662}). Here we define the \emph{trivial module} over $kA$, denoted by $k_A$, to be the one-dimensional vector space over $k$ on which both $(A,\cdot)$ and $(A,\circ)$ act trivially, in other words, on which $\Lambda_{A}$ acts trivially.

\begin{thm}\label{thm:Brauer'}Suppose that $A$ is finite and that $\mathrm{char}(k)$ is a prime $p$. Then $|A|$ is a power of $p$ if and only if $k_A$ is the only simple $kA$-module.
\end{thm}
\begin{proof} If $|A|$ is a power of $p$, then so is $|\Lambda_A| = |A|^2$ and it is known that $k_A$ is the only simple $k\Lambda_A$-module in this case (see \cite[Proposition 6.2.1]{Webb}). Thus, in particular $k_A$ is the only simple $kA$-module.

\vspace{3mm}

If $|A|$ is not a power of $p$, then we know that there exists a non-trivial simple $k(A,\circ)$-module $S$ (see Theorem 7.1 of \href{https://arxiv.org/abs/2405.08662v1}{arXiv:2405.08662v1}). We regard $S$ as a $kA$-module by letting $(A,\cdot)$ act trivially (recall Example \ref{ex:circ}). Then, clearly $S$ is non-trivial and simple as a $kA$-module as well.
\end{proof}

In Theorem \ref{thm:Brauer'}, in the direction where we assume that $|A|$ is not a power of $p$, it is possible to exhibit simple $kA$-modules for which both the actions of $(A,\cdot)$ and $(A,\circ)$ are non-trivial when the group
\[ A/A*A = (A/A*A,\cdot) = (A/A*A,\circ)\]
admits a non-trivial one-dimensional $k(A/A*A)$-module $S$, for example. To be precise, let $S\otimes_k S$ be the $kA$-module as constructed in Example \ref{ex:lift'}. For any $a\in A$, let $\ell_a\in k^\times$ be such that $a\cdot (A*A) = a\circ (A*A)$ acts on $S$ via scalar multiplication by $\ell_a$. By the definition in Example \ref{ex:lift'}, the actions of $(A,\cdot)$ and $(A,\circ)$ on $S\otimes_kS$ are given by
\begin{align*}
a\cdot (s\otimes t) & = (\ell_as)\otimes t = \ell_a(s\otimes t)\\
a\circ (s\otimes t) & = s \otimes (\ell_at) = \ell_a(s\otimes t)
\end{align*}
for all $a\in A$ and $s,t\in S$. Since $S$ is a non-trivial $k(A/A*A)$-module, there exists $a\in A$ such that $\ell_a\neq 1$. This implies that both $(A,\cdot)$ and $(A,\circ)$ act non-trivially on $S\otimes_kS$. Here $S\otimes_kS$ is clearly a simple $kA$-module because it is one-dimensional.

\begin{example}Consider the brace $A = (\bZ/q\bZ\times \bZ/q'\bZ,+,\circ)$, where
\[ \begin{pmatrix}a_1\\a_2\end{pmatrix} \circ \begin{pmatrix} b_1\\ b_2\end{pmatrix} = 
\begin{pmatrix}
a_1 + \lambda^{a_2}b_1\\ a_2+b_2\end{pmatrix}\]
for all $\left(\begin{smallmatrix}a_1\\a_2\end{smallmatrix}\right),\left(\begin{smallmatrix}b_1\\b_2\end{smallmatrix}\right)\in \bZ/q\bZ\times \bZ/q'\bZ$, $q,q'$ are primes for which $q\equiv1\pmod{q'}$, and $\lambda \in \mathbb{Z}/q\mathbb{Z}$ is an element of multiplicative order $q'$ (c.f. \cite[Theorem]{pq}). It is straightforward to see that
\[ A*A = \mathbb{Z}/q\mathbb{Z}\times\{0\}\mbox{ and }A/A*A\simeq \mathbb{Z}/q'\mathbb{Z}.\]
Thus, when $k$ is a splitting field of $A/A*A$ of characteristic coprime to $q'$, for example, there is a non-trivial one-dimensional $k(A/A*A)$-module, and we can use the above method to construct a simple $kA$-module for which the actions of both $(A,\cdot)$ and $(A,\circ)$ are non-trivial.
\end{example}

\section{Some concrete examples}\label{sec:ex}

A representation $(V,\beta,\rho)$ of $A$ is clearly irreducible if either $(V,\beta)$ or $(V,\rho)$ is irreducible as a group representation, and similarly for indecomposable. It is possible, as Examples \ref{ex:Z/2xZ/2} and \ref{ex2} below show, respectively, that
\begin{enumerate}[(1)]
    \item both $(V,\beta)$ and $(V,\rho)$ are reducible, yet $(V,\beta,\rho)$ is irreducible;
    \item both $(V,\beta)$ and $(V,\rho)$ are decomposable, yet $(V,\beta,\rho)$ is indecomposable though reducible.
\end{enumerate}
Irreducible and indecomposable representations of skew braces are therefore much more difficult to understand than those of groups, even when one has a complete classification of the irreducible representations of both the additive and multiplicative groups.

\begin{example}\label{ex:Z/2xZ/2}
Consider the brace $(\bZ/2\bZ\times \bZ/2\bZ,+,\circ)$, where 
\[ \begin{pmatrix}a_1\\ a_2 \end{pmatrix} \circ \begin{pmatrix} b_1\\b_2\end{pmatrix} = \begin{pmatrix} a_1 + b_1 + a_2b_2 \\ a_2 + b_2\end{pmatrix}\]
for all $\left(\begin{smallmatrix}a_1\\ a_2\end{smallmatrix}\right),\left(\begin{smallmatrix} b_1\\ b_2\end{smallmatrix}\right)\in \bZ/2\bZ\times \bZ/2\bZ$ (c.f. \cite[Proposition 2.4]{p^3}), and note that 
\begin{align*}
 (\bZ/2\bZ\times \bZ/2\bZ,\circ)
 & = \left\langle \begin{pmatrix} 1\\ 1 \end{pmatrix}\right\rangle
  = \left\{ \begin{pmatrix} 0\\ 0 \end{pmatrix},\begin{pmatrix} 1\\ 1 \end{pmatrix},\begin{pmatrix} 1\\ 0 \end{pmatrix},\begin{pmatrix} 0\\ 1 \end{pmatrix}\right\}
\end{align*}
is cyclic. Let $\beta : (\bZ/2\bZ\times \bZ/2\bZ,+) \longrightarrow \GL_2(k)$ be defined by
\[\beta\left(\begin{matrix}1\\0\end{matrix}\right) = \begin{pmatrix} -1& 0 \\ 0 & -1\end{pmatrix}, \,\ \beta\left( \begin{matrix}0\\1\end{matrix}\right) = \begin{pmatrix}0&1\\1 & 0 \end{pmatrix},\]
and let $\rho: (\bZ/2\bZ\times \bZ/2\bZ,\circ)\longrightarrow \GL_2(k)$ be defined by
\[ \rho\left(\begin{matrix}1\\1\end{matrix}\right) = \begin{pmatrix} 0 & -1 \\ 1 & 0 \end{pmatrix}.\]
To verify the relation (\ref{relation}), note that it may be written as
\[ 
\beta\begin{pmatrix}a_2b_2 \\ 0\end{pmatrix} = \rho\left( \begin{matrix}a_1\\a_2\end{matrix}\right) \beta\left( \begin{matrix}b_1\\b_2\end{matrix}\right) \rho\left(\begin{matrix}a_1\\a_2\end{matrix} \right)^{-1}\beta\begin{pmatrix}b_1 \\ b_2\end{pmatrix}
^{-1}.\]
For $a_2 = 0$ or $b_2=0$, this is trivial since $\beta\left(\begin{smallmatrix}*\\0\end{smallmatrix}\right),\rho\left(\begin{smallmatrix}*\\0\end{smallmatrix}\right) = \pm \left(\begin{smallmatrix}1&0  \\ 0 & 1\end{smallmatrix}\right)$. 
For $a_2=1$ and $b_2 =1$, this holds because $\beta\left(\begin{smallmatrix} * \\ 1 \end{smallmatrix}\right) = \pm \left(\begin{smallmatrix} 0 & 1 \\ 1 & 0 \end{smallmatrix}\right),\, 
 \rho\left(\begin{smallmatrix}  * \\ 1 \end{smallmatrix}\right) = \pm \left(\begin{smallmatrix}  0 & -1\\ 1 & 0 \end{smallmatrix}\right)$ and
\[ \beta\begin{pmatrix} 1\\0\end{pmatrix} 
= \begin{pmatrix} 0 & - 1\\ 1 &0 \end{pmatrix} \begin{pmatrix}0 & 1 \\ 1 & 0 \end{pmatrix}
\begin{pmatrix} 0 & - 1\\ 1 & 0 \end{pmatrix}^{-1}\begin{pmatrix}0 & 1 \\ 1 & 0 \end{pmatrix}^{-1}.\] 
Thus $(k^2, \beta,\rho)$ is a representation of the brace under consideration.
 
 \vspace{3mm}
 
Suppose that $\mathrm{char}(k)\neq 2$ and that $-1$ has a square root $i$ in $k$. Then, by computing eigenspaces of the generators of $\mathrm{Im}(\beta)$ and $\mathrm{Im}(\rho)$, we see that the one-dimensional $\beta$-invariant subspaces are exactly
\[\mathrm{span}_{k} \begin{pmatrix}1\\1\end{pmatrix},\,\ \mathrm{span}_{k}\begin{pmatrix}-1\\1\end{pmatrix},\]
and the one-dimensional $\rho$-invariant subspaces are exactly
\[\mathrm{span}_{k} \begin{pmatrix}i\\1\end{pmatrix},\,\ \mathrm{span}_{k}\begin{pmatrix}-i\\1\end{pmatrix}.\]
Thus, the group representations $(k^2,\beta)$ and $(k^2,\rho)$ are both reducible (in fact completely reducible because $\left(\begin{smallmatrix}1\\1 \end{smallmatrix}\right),\left(\begin{smallmatrix}-1\\1 \end{smallmatrix}\right)$ and $\left(\begin{smallmatrix}i\\1 \end{smallmatrix}\right),\left(\begin{smallmatrix}-i\\1 \end{smallmatrix}\right)$ are bases of $k^2$). Since the four one-dimensional subspaces above are pairwise distinct, there are no one-dimensional $\varphi_{(\beta,\rho)}$-invariant subspaces, whence the brace representation $(k^2,\beta,\rho)$ is irreducible.
\end{example}

\begin{example}\label{ex2}Let $(\bZ/p\bZ\times \bZ/2\bZ,+,\circ)$ be the same brace in Example \ref{ex:beta}, where $p$ is an odd prime. Note that the multiplicative group is a semidirect product
\begin{align*}
(\bZ/p\bZ\times \bZ/2\bZ,\circ) 
&
= \left\{\begin{pmatrix}0\\0\end{pmatrix},
\begin{pmatrix}1\\0\end{pmatrix},\dots,
\begin{pmatrix}p-1\\0\end{pmatrix}
 \right\}\rtimes\left\{\begin{pmatrix}0\\0\end{pmatrix},\begin{pmatrix}0\\1\end{pmatrix}\right\},
\end{align*}
where $\left(\begin{smallmatrix}0\\1\end{smallmatrix}\right) \circ \left(\begin{smallmatrix}x\\0\end{smallmatrix}\right)\circ \overline{\left(\begin{smallmatrix}0\\1\end{smallmatrix}\right)}=-\left(\begin{smallmatrix}x\\0\end{smallmatrix}\right)$ for all $x\in \mathbb{Z}/p\mathbb{Z}$.
Suppose that $\mathrm{char}(k) = p$. Let $\beta : (\bZ/p\bZ\times \bZ/2\bZ,+) \longrightarrow \GL_3(k)$ be defined by
\[\beta\begin{pmatrix}1\\0\end{pmatrix} = \begin{pmatrix}1&0&0\\0&1&1\\0 & 0 & 1\end{pmatrix},\,\ \beta\begin{pmatrix}0\\1\end{pmatrix} =\begin{pmatrix} 1 & 0 & 0 \\ 0 & 1 & 0 \\ 0 & 0 & 1\end{pmatrix},\]
and let $\rho: (\bZ/p\bZ\times \bZ/2\bZ,\circ)\longrightarrow \GL_3(k)$ be defined by
\[ \rho\begin{pmatrix} 1 \\ 0 \end{pmatrix}=\begin{pmatrix}1 & 0 & 0 \\ 1 & 1 & 0 \\ 0 & 0 & 1 \end{pmatrix},\,\ \rho\begin{pmatrix} 0 \\ 1 \end{pmatrix}   = \begin{pmatrix} 1 & 0 & 0 \\ 0 & -1&0 \\ 0 & 0 & 1\end{pmatrix}.\]
To verify the relation (\ref{relation}), note that it may be written as
\[ \beta\begin{pmatrix} (-1)^{a_2}b_1\\b_2\end{pmatrix}
= \rho\begin{pmatrix}a_1\\a_2\end{pmatrix}\beta\begin{pmatrix}b_1\\b_2\end{pmatrix}\rho\begin{pmatrix}a_1\\a_2\end{pmatrix}^{-1}.\]
We may put $b_2=0$ because $\beta\left(\begin{smallmatrix}0\\1\end{smallmatrix}\right)$ is the identity matrix. Since
\[ \begin{pmatrix}a_1\\a_2\end{pmatrix}
 = \begin{pmatrix} 0\\a_2 \end{pmatrix}\circ
 \begin{pmatrix} (-1)^{a_2}a_1\\0\end{pmatrix}\]
and $\rho\left(\begin{smallmatrix}*\\0\end{smallmatrix}\right)$ commutes with elements of $\mathrm{Im}(\beta)$, the value of $a_1$ does not matter so we may put $a_1=0$. We also see that the relation is trivial for $a_2=0$. 
For $b_2=0,\, a_1=0$, and $a_2=1$, in matrix form the relation becomes
\begin{align*}
\begin{pmatrix}
1 & 0 & 0\\
0 & 1 & -b_1\\
0 & 0 & 1
\end{pmatrix}
= \begin{pmatrix}
1 & 0 & 0 \\ 0 & -1 & 0 \\ 0 & 0 & 1
\end{pmatrix}
\begin{pmatrix}
1 & 0 & 0 \\
0 & 1 & b_1\\
0 & 0 & 1
\end{pmatrix}
 \begin{pmatrix}
1 & 0 & 0 \\ 0 & -1 & 0 \\ 0 & 0 & 1
\end{pmatrix}^{-1},
\end{align*}
which is easily seen to be true. Thus, the triplet $(k^3,\beta,\rho)$ is a representation

\noindent of the brace under consideration.

\vspace{3mm}

To simplify notation, let $e_1,\, e_2,\, e_3$ denote the standard basis vectors of $k^3$.
The group representation $(k^3,\beta)$ is decomposable because
\[ k^3 = \mathrm{span}_k(e_1) \oplus \mathrm{span}_k(e_2,e_3)\]
and the two direct summands are $\beta$-invariant. Similarly, the group representation $(k^3,\rho)$ is also decomposable because
\[ k^3 = \mathrm{span}_k(e_1,e_2) \oplus \mathrm{span}_k(e_3)\]
and the two direct summands are $\rho$-invariant. A computation of eigenspaces yields that $\mathrm{span}_k(e_2)$ is the only one-dimensional $\varphi_{(\beta,\rho)}$-invariant subspace of $k^3$. Thus, if $(k^3,\beta,\rho)$ were decomposable, then necessarily
\[ k^3 = \mathrm{span}_k(e_2) \oplus U \]
for some $\varphi_{(\beta,\rho)}$-invariant subspace $U$. But for any $u\in U$, both of 
\[\begin{cases}
 \beta\left( \begin{smallmatrix}1\\0\end{smallmatrix}\right)(u) - u = u_3e_2, \\[4pt]
\rho\left( \begin{smallmatrix}1\\0\end{smallmatrix}\right)(u) -u =u_1e_2, 
\end{cases}\mbox{ where }u = \begin{pmatrix}u_1\\u_2\\u_3\end{pmatrix},\]
must lie in $U$ because $U$ is $\varphi_{(\beta,\rho)}$-invariant. Since $\mathrm{span}_k(e_2)\cap U=\{0\}$, we get that $u_1=u_3=0$ and so $U\subseteq\mathrm{span}_k(e_2)$, which is a contradiction. This shows that $(V,\beta,\rho)$ is indecomposable, though it is reducible since $\mathrm{span}_k(e_2)$ is a $\varphi_{(\beta,\rho)}$-invariant subspace.
\end{example}

In the case that $A$ is finite,  for any irreducible representation $(V,\beta,\rho)$ of $A$, we know from Proposition \ref{prop:simple} that $(V,\beta)$ must be completely reducible. The next example shows that $(V,\rho)$ need not be completely reducible in general.

\begin{example}\label{ex:counterexample}Consider the skew brace $(S_3,\cdot,\circ)$, where $\circ$ denotes the operation arising from the exact factorization $S_3= \langle(1\, 2)\rangle A_3$, namely
\[ \sigma\circ \tau = \sigma_1\tau\sigma_2\]
for all $\sigma,\tau\in S_3$ with $\sigma = \sigma_1\sigma_2$ for $\sigma_1\in \langle(1\, 2)\rangle$ and $\sigma_2\in A_3$ (c.f. \cite[Example 1.6]{GV}). Note that the multiplicative group is a direct product 
\[ (S_3,\circ) = \langle (1\, 2\,3)\rangle \times \langle (1\, 2)\rangle = A_3 \times \{(1),(1\,2)\}.\]
Suppose that $\mathrm{char}(k) = 2$. Let $\beta : (S_3,\cdot)\longrightarrow\mathrm{GL}_2(k)$ be defined by
\[ \beta(1\, 2\, 3)= \begin{pmatrix} 0 & 1 \\ 1 & 1\end{pmatrix},\,\ 
\beta(1\, 2) = \begin{pmatrix} 1 & 1 \\ 0 & 1 \end{pmatrix},\]
and let $\rho: (S_3,\circ)\longrightarrow\mathrm{GL}_2(k)$ be defined by
\[ \rho(1\, 2\, 3)= \begin{pmatrix} 1 & 0 \\ 0 & 1\end{pmatrix},\,\ 
\rho(1\, 2) = \begin{pmatrix} 1 & 1 \\ 0 & 1 \end{pmatrix}.\]
To verify the relation (\ref{relation}), note that it may be written as
\[ \beta(\sigma_1)\beta(\tau)\beta(\sigma_1)^{-1} = \rho(\sigma) \beta(\tau)\rho(\sigma)^{-1},\]
where $\sigma = \sigma_1\sigma_2$ with $\sigma_1 \in \langle (1\, 2)\rangle$ and $\sigma_2\in A_3$. Since $\sigma_1\sigma_2 = \sigma_1\circ \sigma_2$ and $\rho$ is trivial on $A_3$, it is enough to consider $\sigma = \sigma_1$. 
But then the relation becomes trivial because $\beta(\sigma_1) = \rho(\sigma_1)$. Thus, the triplet $(k^2,\beta,\rho)$ is a representation of the skew brace under consideration. 

\vspace{3mm}

Clearly the group representation $(k^2,\beta)$ is irreducible, so in particular the skew brace representation $(k^2,\beta,\rho)$ is also irreducible, but $(k^2,\rho)$ is not completely reducible (it is reducible and indecomposable).\end{example}

In the case that $A$ is trivial, one can identity a representation $(V,\alpha)$ of the group $(A,\cdot)$ with the representation $(V,\alpha,\alpha)$ of the skew brace $A$, and this identification is equivalence preserving. For any $f\in \GL(V)$, clearly
\[ \alpha_f : (A,\cdot) \longrightarrow \GL(V);\,\ \alpha_f(a) = f\alpha(a)  f^{-1} \]
is a representation of $(A,\cdot)$ that is equivalent to $(V,\alpha)$. Note that $(V,\alpha,\alpha_f)$ is a representation of $A$ if and only if
\[ \alpha(a)\alpha(b)\alpha(a)^{-1} =\alpha_f(a) \alpha(b) \alpha_f(a)^{-1}\]
for all $a,b\in A$, or equivalently by replacing $a$ with $a^{-1}$, if the commutator 
\[ \alpha_f(a)\alpha(a)^{-1} = f\alpha(a) f^{-1} \alpha(a)^{-1}\]
centralizes $\mathrm{Im}(\alpha)$ for all $a\in A$. But $(V,\alpha,\alpha)$ and $(V,\alpha,\alpha_f)$ are equivalent if and only if $\alpha_f = \alpha$. Thus, if we can find $f\in \GL(V)$ such that
\begin{enumerate}[(1)]
\item $f\alpha(a) f^{-1} \alpha(a)^{-1}$ centralizes $\mathrm{Im}(\alpha)$ for all $a\in A$;
\item $f$ does not centralize $\mathrm{Im}(\alpha)$;
\end{enumerate}
then even though $(V,\alpha)$ and $(V,\alpha_f)$ are equivalent as representations of the group $(A,\cdot)$, the triplets $(V,\alpha,\alpha)$ and $(V,\alpha,\alpha_f)$ are not equivalent as representations of the skew brace $A$. We shall give a concrete example below. It follows that the equivalence of representations of skew braces is much more difficult to classify than that of representations of groups. 

\begin{example}Consider the symmetric group $S_3 = (S_3,\cdot)$ on three letters. Let $\alpha: S_3\longrightarrow \GL_2(k)$ be the (irreducible) representation defined by
\[ \alpha(1\, 2\, 3) = \begin{pmatrix} 0 & -1\\ 1 & -1\end{pmatrix},\,\ \alpha(1\, 2) = \begin{pmatrix} - 1 & 1 \\ 0 & 1 \end{pmatrix}.\]
Suppose that $\mathrm{char}(k)\neq 2,3$, and take $f\in \GL_2(k)$ to be the matrix
\[ f = \begin{pmatrix} -1 & 2 \\ -2 & 1 \end{pmatrix},\]
which is invertible since $\mathrm{char}(k)\neq3$. Then $f$ satisfies (1) above because 
\begin{align*}
 f\alpha(1\,2\,3)f^{-1} \alpha(1\,2\,3)^{-1} & = \begin{pmatrix}1 & 0 \\ 0 & 1\end{pmatrix},\\
 f\alpha(1\,2) f^{-1} \alpha(1\,2)^{-1} & = \begin{pmatrix}-1 & 0 \\ 0 & -1\end{pmatrix},
 \end{align*}
and $f$ does not commute with $\alpha(1\,2)$ since $\mathrm{char}(k)\neq2$, so it satisfies (2).
\end{example}

\section*{Acknowledgments} 

This project started from discussions at the symposium (open) ``Research on finite groups, algebraic combinatorics, and vertex algebras" held in December 2023 at the Research Institute for Mathematical Sciences, an International Joint Usage/Research Center located in Kyoto University.

\vspace{3mm}

This work is supported by JSPS KAKENHI Grant Number 24K16891.

\vspace{3mm}

The authors would like to thank Kyoichi Suzuki for useful discussions on modular representation theory of finite groups, as well as Thomas Letourmy and Leandro Vendramin for pointing out Remark \ref{key remark}.


\begin{thebibliography}{99}

\bibitem{pq}
E. Acri and M. Bonatto, \emph{Skew braces of size $pq$}, Comm. Algebra 48 (2020), no. 5, 1872--1881.


\bibitem{p^3}
D. Bachiller, \emph{Classification of braces of order $p^3$}, J. Pure Appl. Algebra 219 (2015), no. 8, 3568--3603.

\bibitem{perfect}
F. Ced\'{o}, A. Smoktunowicz, and L. Vendramin, \textit{Skew left braces of nilpotent type}, Proc. Lond. Math. Soc. (3) 118 (2019), no. 6, 1367--1392.

\bibitem{GV}
L. Guarnieri and L. Vendramin, \textit{Skew braces and the Yang-Baxter equation}, Math. Comp. 86 (2017), no. 307, 2519--2534. 

\bibitem{factor}
E. Jespers, \L. Kubat, A. Van Antwerpen, and L. Vendramin, \textit{Factorizations of skew braces}, Math. Ann. 375 (2019), no. 3-4, 1649--1663.

\bibitem{ann}
E. Jespers, A. Van Antwerpen, and L. Vendramin, \textit{Nilpotency of skew braces and multipermutation solutions of the Yang-Baxter equation}, Commun. Contemp. Math. 25 (2023), no. 9, Paper No. 2250064, 20 pp.

\bibitem{op}
A. Koch and P. J. Truman, \textit{Opposite skew left braces and applications}, J. Algebra 546 (2020), 218--235.

\bibitem{isoclinism}
T. Letourmy and L. Vendramin, \emph{Isoclinism of skew braces}, Bull. Lond. Math. Soc. 55 (2023), no. 6, 2891--2906.

\bibitem{Schur cover}
T. Letourmy and L. Vendramin, \emph{Schur covers of skew braces}, J. Algebra 644 (2024), 609--654.


\bibitem{Rump}
W. Rump, \textit{Braces, radical rings, and the quantum Yang-Baxter equation}, J. Algebra 307 (2007), no. 1, 153--170. 


\bibitem{Ito}
C. Tsang, \textit{A generalization of Ito's theorem to skew braces}, J. Algebra 642 (2024), 367--399.

\bibitem{Webb}
P. Webb, \emph{A course in finite group representation theory}, Cambridge Studies in Advanced Mathematics, 161. Cambridge University Press, Cambridge, 2016.

\bibitem{Zhu}
H. Zhu, \textit{The construction of braided tensor categories from Hopf braces}, Linear Multilinear Algebra 70 (2022), no. 16, 3171--3188.

\end{thebibliography}
\end{document}